\documentclass[11pt]{amsart}

\usepackage{amsmath,amsthm,amssymb,amsfonts,color}

\usepackage{ stmaryrd }
\usepackage{ mathtools }
\usepackage{ nicefrac }

\usepackage{enumitem}
\newlist{myitemize}{itemize}{3}
\setlist[myitemize,1]{label=,leftmargin=0.2in}
\setlist[myitemize,2]{label=$\rightarrow$,leftmargin=1em}
\setlist[myitemize,3]{label=$\diamond$}
\newcommand{\R}{\mathbb{R}}

\newcommand{\p}{ \partial  }
\DeclareMathOperator{\diam}{diam}

\theoremstyle{plain}
\newtheorem{defi}{Definition}[section]
\newtheorem{prop}[defi]{Proposition}
\newtheorem{teo}[defi]{Theorem}
\newtheorem{cor}[defi]{Corollary}
\newtheorem{lema}[defi]{Lemma}

\newtheorem{remark}[defi]{Remark}

\theoremstyle{definition}

\theoremstyle{remark}

\numberwithin{equation}{section}

\begin{document}

\title[Optimal mass transport for eigenvalue problems.]{An optimal mass transport approach for limits of eigenvalue problems for the fractional $p$-Laplacian.}

\author[L. M. Del Pezzo]{L. M. Del Pezzo}

\address[L. M. Del Pezzo, J. D. Rossi, N. Saintier,  A. Salort]{\newline
\noindent CONICET and Dep. de Matem\'{a}tica, FCEyN, 
Universidad de Buenos Aires, Ciudad Universitaria, Pab 1 (1428), 
Buenos Aires, Argentina.}
\address{\newline
N. Saintier: Instituto de Ciencias - Univ. Nac. Gral Sarmiento, J. M. Gutierrez 1150, C.P. 1613 Los Polvorines - Pcia de Bs. As. - Argentina }
\email{\newline
ldpezzo@dm.uba.ar \newline
jrossi@dm.uba.ar \newline
nsaintier@dm.uba.ar \newline
asalort@dm.uba.ar}
\urladdr{\newline
http://cms.dm.uba.ar/Members/ldpezzo\newline
 http://mate.dm.uba.ar/$\thicksim$jrossi/ }

\author[J. D. Rossi]{J. D. Rossi}


\author[N. Saintier]{N. Saintier}


\author[A. Salort]{A. Salort}




\begin{abstract}
	We find interpretation using optimal mass transport theory for 
	eigenvalue problems 
	obtained as limits of the eigenvalue problems for the fractional 
	$p-$Laplacian operators as $p\to +\infty$. 
	We deal both with Dirichlet and Neumann boundary conditions.
\end{abstract}

\date{\today}

\maketitle

\section{Introduction}
	Our main goal in this paper is to use tools from mass transport 
	theory to study eigenvalue problems that are obtained 
	taking limits as $p\to +\infty$ in eigenvalue problems that 
	involve fractional spaces $W^{s,p}$ (with $0<s<1$ and $1<p\leq+\infty$).
	We deal both with Dirichlet and Neumann boundary conditions.
	
	Along this paper we take $U$ a smooth bounded domain in $\R^n$, 
	$1<p<+\infty$ and $0<s<1$. We also fix $d(\cdot,\cdot)$ a 
	distance in $\R^n$ equivalent to the Euclidean one.
	
	\medskip
	
	Let $\lambda_{s,p}^D$ 
	be the first eigenvalue of the fractional $p$-Laplacian of order 
	$s$ in	$U$   with Dirichlet boundary conditions, 
	that is, let us consider
	$$
		\lambda^D_{s,p}
		\coloneqq\inf\,\left
		\{ [u]^p_{s,p}: u\in \widetilde{W}^{s,p}(U), \int_U |u|^p
		\,dx =1 \right\}
	$$
	where 
	$$
		[u]^p_{s,p}\coloneqq
		 \int_{\R^n}\int_{\R^n} \frac{|u(x)-u(y)|^p}{|x-y|^{n+sp}}
			\,dxdy
	$$ 
	is the semi-norm of $W^{s,p}(\R^n)$ and 
	$$
		\widetilde{W}^{s,p}(U) \coloneqq
		\left\{u \in W^{s,p}(\R^n) \, : \, u=0 \mbox{ in } \R^n 
		\setminus U \right\}.
	$$

	For this problem Lindgren and Lindqvist in \cite{LL} proved that 
	$$
		\Lambda^D_{s,\infty}\coloneqq
		\lim_{p\to + \infty}\left(\lambda_{s,p}^D\right)^{1/p}  
		= \frac{1}{R^s} 
	$$
	where 
	\[
		R\coloneqq\max_{x\in \overline{U}} \mbox{dist}(x,\p U) = 
		\max_{x\in \overline{U}} \min_{y\in 
		\partial U} |x-y|.
	\]
	
	Moreover, via a subsequence, the eigenfunctions $u_p$ suitable normalized 
	(a minimizer for $\lambda_{s,p}^D$)  converge uniformly to a minimizer for
	 $\Lambda_{s,\infty}^D.$

	Our first purpose in this work is to relate $\Lambda^D_{s,\infty}$ 
	to an optimal mass transport problem with cost 
	function $c(x,y)=|x-y|^s.$ We prove the following result:

	\begin{teo} \label{teo.Dir.Intro} 
		There holds that
			$$
				\frac{1}{\Lambda_{s,\infty}^D} 
 				= \sup_{\mu\in P(\overline{U})} W_s(\mu,P(\p U)) ,
 			$$
			where 
		\[
			W_s(\mu,\nu)\coloneqq\inf_{\pi} \int_U\int_{U} |x-y|^s\,d
			\pi (x,y).
		\] 
		Here $P(A)$ is the set of probability measures on $A$  and 
		$\pi\in P(\overline{U}\times \overline{U}) $ is a measure with
		marginals $\mu$ and $\nu$.
	\end{teo}

	Note that $W_s(\mu,\nu)$ is the total cost when we have to transport the
	measure $\mu$ onto $\nu$ using as cost for transporting one unit of
	mass from position $x$ to position $y$ the Euclidean distance to the power 
	$s$, that is $|x-y|^s$. 
	We refer to \cite{Villani} and to Section \ref{sect-MK} for precise 
	definitions, notations and properties of optimal mass transport theory. 
	Hence, our result says that the eigenvalue $\Lambda_{s,\infty}^D$ is 
	related to the problem of finding a probability measure supported inside 
	$U$, $\mu$, that is 
	far (in terms of the transport cost) from the set of probability measures 
	supported on the boundary, $\partial U$. One easy solution to this problem 
	is the following: take $B_R(x_0)$ a ball with maximum radius $R$ inside $U$ 
	and let $y_0 \in \partial U \cap \partial B_R(x_0)$ 
	(there exists such $y_0$ due to the maximality of $R$). 
	Then, $\mu = \delta_{x_0}$ 
	(with $\nu = \delta_{y_0}$) solves $\sup_{\mu\in P(\overline{U})} 
	W_s(\mu,P(\p U))$. Observe that from Theorem \ref{teo.Dir.Intro} we 
	can recover that $\Lambda^D_{s,\infty}		= \nicefrac{1}{R^s}$.

	\medskip

	Now, let us turn our attention to the case of the first nontrivial 
	eigenvalue for Neumann boundary conditions, i.e., let us consider
 	\[
		\lambda^N_{s,p} \coloneqq 
		\inf\left\{\llbracket u\rrbracket^p_{s,p}\colon u\in \mathcal{C}
		\right\},
	\]
	where 
	\begin{equation}\label{DefLambdapsNeu}
	\llbracket u\rrbracket^p_{s,p}\coloneqq
		\int_U\int_U \frac{ |u(x)-u(y)|^p}{d(x,y)^{n+sp}}\,dxdy
	\end{equation}
	and
	\[
		\mathcal{C}\coloneqq\left\{u\in W^{s,p}(U)\colon
		\|u\|_{L^p(U)} =1, \, \int_U |u|^{p-2} u \,dx =0 \right\}.
	\]
	For this problem, in the case $d(x,y)=|x-y|,$ 
	Del Pezzo and Salort in \cite{DelpeSal} proved that
	$$
		\Lambda^N_{s,\infty} \coloneqq
		\lim_{p\to +\infty}\left(\lambda^N_{s,p}\right)^{1/p}  
		= \frac{2}{(\diam(U))^s} 
	$$
	where $\diam(U)$ is the extrinsic diameter, i.e.
	\[
		\diam(U) \coloneqq \max_{x,y\in\overline{U}} |x-y|.
	\]

	Their proof actually extends to the case in which we consider
	$\llbracket u\rrbracket^p_{s,p}$ with $d(x,y)$ any distance as above  
	(for instance, for the geodesic distance in $U$).
	In this case, it holds that
	\begin{equation}\label{eq:1ercaso}
		\Lambda_{s,\infty}^N
		= \lim_{p\to +\infty}\left(\lambda_{s,p}^N\right)^{1/p} 
		= \frac{2}{(\diam_d(U))^s}
	\end{equation}
	where $\diam_d(U)$ is the diameter of $U$ according to $d$, 
	that is 
	\[
		\diam_d(U) = \max_{x,y\in\overline{U}} d(x,y).
	\]
	Moreover, as happens for the Dirichlet problem, 
	via a subsequence the normalized eigenfunctions $u_p$ (a minimizer for 
	$\lambda_{s,p}^N$) converge uniformly to a minimizer for
	 $\Lambda_{s,\infty}^N.$
	 
	 In order to introduce the mass transport interpretation we need the 
	 following notations. We denote by $M(\overline{U})$ the space of finite 
	 Borel measures over 
	$\overline{U}$. Given 
	$\sigma\in M(\overline{U})$, we denote its positive and negative part by 
	$\sigma^+$ and $\sigma^-$ so that $\sigma=\sigma^+-\sigma^-$, and 
	$|\sigma|=\sigma^++\sigma^-$. Then we have,

	\begin{teo} \label{teo.Neu.Intro}
		There holds
		\begin{equation}\label{MK.Neu}
 			\frac{2}{\Lambda^N_{s,\infty}}
 				=  \max\left\{
 				W_s(\sigma^+,\sigma^-)\colon
 				{\sigma\in M(\overline{U}), \, 
 				\sigma^+(\overline{U})=\sigma^-(\overline{U})=1} 
 				 \right\}
		\end{equation}
		where $W_s$ is as in Theorem \ref{teo.Dir.Intro}.
	\end{teo}

	Here we relate $\Lambda^N_{s,\infty}$ to the problem of finding two 
	probability measures, $\sigma^+$ and $\sigma^-$, supported in 
	$\overline{U}$, such that the cost of transporting one into the 
	other is maximized. To 
	obtain a solution to this problem one can argue as follows: take two points $x_0$ and $y_0$ in $\overline{U}$ that realize the diameter, that 
	is, we have $d(x_0,y_0)= \diam_d (U)$. Then take 
	$\sigma^+ = \delta_{x_0}$ and $\sigma^- = \delta_{y_0}$ as a solution to 
	$\max\left\{W_s(\sigma^+,\sigma^-)\colon
 	{\sigma\in M(\overline{U}), \, \sigma^+(\overline{U})=
 	\sigma^-(\overline{U})=1} \right\}$. Note that we can recover 
 	\eqref{eq:1ercaso} from Theorem \ref{teo.Neu.Intro}.
	
	\medskip
	
	A different concept of Neumann boundary condition for fractional operators
	was recently introduced in \cite{DRV}. More precisely, for 
	$(-\Delta)_p^s$ the fractional $p-$Laplacian given by
	\[
		(-\Delta)_p^su(x)=\mbox{P.V.}\int_{\mathbb{R}^n}
		\dfrac{|u(x)-u(y)|^{p-2}(u(x)-u(y))}{d(x,y)^{n+sp}}\, dy
	\]
	(the symbol P.V. stands for the principal value of the integral), 
	we consider the following non-local non-linear fractional 
	normal derivative
	\[
		\mathcal{N}_{s,p}u(x)=\int_{U}
		\dfrac{|u(x)-u(y)|^{p-2}(u(x)-u(y))}{d(x,y)^{n+sp}}\, dy\qquad 
		x\in\mathbb{R}^n\setminus\overline{U}.
	\]
	Associated with this operator, we consider the following eigenvalue problems
	\begin{equation}\label{eq:naut.int}
		\begin{cases}
			(-\Delta)_p^su = 	\lambda |u|^{p-2}u &\text{ in } U,\\
			\mathcal{N}_{s,p} u= 0 	&\text{ in }
			\mathbb{R}^n\setminus\overline{U}.
		\end{cases}
	\end{equation}
	
	Before stating our main result concerning   these problems,  
	we need to introduce some notations.
	Let $\mathcal{W}^{s,p}(U)$ be the set of measurable functions with finite  
	$$
			\|u\|_{\mathcal{W}^{s,p}(U)}^p\coloneqq
			\|u\|_{L^p(U)}^p+ \mathcal{H}_{s,p}(u),
	$$
	where
	\[
			\mathcal{H}_{s,p}(u)\coloneqq 
			\iint_{\mathbb{R}^{2n}\setminus{(U^c)^2}}
			\dfrac{|u(x)-u(y)|^{p}}{d(x,y)^{n+ps}}
			\, dxdy,
	\]
	and $(U^c)^2=U^c\times U^c.$
	Let us also introduce
	\[
		\mathcal{H}_{s,\infty}(u)\coloneqq\sup\left\{
			\dfrac{|u(x)-u(y)|}{d(x,y)^s}\colon (x,y)
			\in\mathbb{R}^{2n}\setminus(U^c)^2\right\}.
	\]
	
	Then, for \eqref{eq:naut.int} we have the following result.
	\begin{teo}\label{teo:oautovalointro}
     	The first non-zero eigenvalue of \eqref{eq:naut.int} is given by
      		$$
			 \lambda_{s,p}=
           		 \inf\left\{
           		 \dfrac{\mathcal{H}_{s,p}(v)}{2\|v\|^p_{L^p(U)}}
           		 \colon 
           		 v\in\mathcal{W}^{s,p}(U)\setminus\{0\}, 
            		\int_U |v|^{p-2}v\, dx=0\right\}.
      		$$
		Concerning the limit as $p\to +\infty$ of these eigenvalues we have
		\[
			\lim_{p\to+ \infty}(\lambda_{s,p})^{\nicefrac1p}=
			\dfrac{2}{(\diam_d(U))^s}=\Lambda_{s,\infty}\coloneqq
			\inf\left\{\dfrac{\mathcal{H}_{s,\infty}(u)}
			{\|u\|_{L^{\infty}(U)}}\colon u\in\mathcal{A}\right\},
		\]
		where 
		\[
			\mathcal{A}\coloneqq\left\{v\in\mathcal{W}^{s,\infty}(U)
			\setminus\{0\}
			\colon \displaystyle\sup_{x\in U} u(x) +\inf_{x\in U} u(x)=0
			\right\}.
		\]
		Moreover, if $u_p$ is a  minimizer of $\lambda_{s,p}$ normalized by 
		$\|u_p\|_{L^p(U)}=1,$
		then, up to a subsequence, $u_p$ converges in $C(\overline{U})$
		to some minimizer $u_{\infty}\in W^{s,\infty}(U)$ of 
		$\Lambda_{s,\infty}^N.$
	\end{teo}
	
	Note that, since the limit of $(\lambda_{s,p})^{\nicefrac1p}$, 
	$\Lambda_{s,\infty}$, coincides with
	$\Lambda_{s,\infty}^N$ (given in \eqref{eq:1ercaso}), we get the same 
	interpretation in terms of optimal mass transportation given in 
	Theorem \ref{teo.Neu.Intro}.
	
	\medskip
	
	To end this introduction, let us briefly comment on previous results.
	The limit as $p\to +\infty$ of the first eigenvalue 
	$\lambda_{p}^D$ of the usual local
	$p$-Laplacian with Dirichlet boundary condition was studied in 
	\cite{JLM,JL}, (see also \cite{BK} for an anisotropic version). 
	In those papers 
	the authors prove that
	$$
		\lambda_{\infty}^D\coloneqq
		\lim_{p\to +\infty}\left(\lambda_{p}^D\right)^{1/p}=
 		\inf \left\{
 		\displaystyle \frac{ \displaystyle \|\nabla v\|_{L^\infty (U)}}
	      {\displaystyle \|v\|_{L^\infty (U)} }\colon
	      v\in W^{1,\infty}_0 (\Omega)\right\}
		= \frac{1}{R},
	$$
	where $R$ is, as before, the largest possible radius of a ball contained 
	in $U$.
	In addition, the authors show the existence of extremals, i.e. 
	functions where the above infimum is attained.
	These extremals can be constructed taking the limit as 
	$p\to +\infty$ in the eigenfunctions of the $p-$Laplacian 
	eigenvalue problems (see \cite{JLM}) and are viscosity solutions of 
	the following eigenvalue problem 
	(called the infinity eigenvalue problem in the literature):
	\begin{equation*}
		\begin{cases}
			\min \left\{|D u|-\lambda_{\infty}^D u,\, 
			\Delta_{\infty} u \right\}=0 &\text{in }U,\\
			u=0& \mbox{on } \p U.
		\end{cases}
	\end{equation*}
	The limit operator
	$\Delta_{\infty}$ that appears here
	is the $\infty$-Laplacian given by
	$\Delta_\infty u = -\langle D^2u Du, Du \rangle.$
	Remark that solutions to 
	$\Delta_p v_p =0$ 
	with a Dirichlet data $v_p=f$ on $\partial U$ converge as 
	$p\to +\infty$ to the viscosity solution to $\Delta_\infty v =0$ 
	with $v=f$ on $\partial U$, 
	see \cite{ACJ,BBM,CIL}. 
	This operator appears naturally when one considers absolutely minimizing 
	Lipschitz extensions in $U$ of a boundary data $f,$ 
	see \cite{A,ACJ,Jensen}.

	Recently in \cite{CDJ}, the authors relate 
	$\lambda_{\infty}^{D}$ with the 
	Monge-Kantorovich distance $W_1$. Recall that the Monge-Kantorovich distance 
	$W_1(\mu,\nu)$ between two probability measures $\mu$ and $\nu$ over 
	$\overline{U}$ is defined by
	\begin{equation} \label{DefW1}
 		W_1(\mu,\nu)\! \coloneqq 
 		\max\left\{
 		\int_U v\,(d\mu-d\nu)
 		\colon v\in W^{1,\infty}(U),\,\|\nabla v\|_{L^\infty(U)}\le 1 
 		\right\}.
	\end{equation}
	It was proved in \cite{CDJ} that
	\begin{equation*} 
 		\frac1{\lambda_{\infty}^D} = \sup_{\mu\in P(U)} W_1(\mu,P(\p U)).
	\end{equation*}
	Notice that this result is the analogous to Theorem \ref{teo.Dir.Intro} in 
	the local case.
	
	For the Neumann problem for the local $p-$Laplacian we refer to 
	\cite{EKNT, RoSaint} where the authors prove the local analogous to Theorem 
	\ref{teo.Neu.Intro}. In this local case the distance that appears  in the 
	limit is the geodesic distance inside $U$. This is, in contrast with the 
	non-local case studied here, where we can consider any distance $d$ 
	equivalent to the Euclidean one, see \eqref{DefLambdapsNeu}.
	
	For limits as $p\to +\infty$ in non-local $p-$Laplacian problems and its 
	relation with optimal mass transport we refer to \cite{Jilka}. Eigenvalue 
	problems were not considered there.

	The case of a Steklov boundary condition has also been investigated 
	recently. 
	Indeed, the authors in \cite{GMPR} (see also \cite{Le} for a slightly 
	different problem) studied the behavior as $p\to+ \infty$ of the so-called 
	variational eigenvalues $\lambda_{k,p}^S$, $k\ge 1$, of the $p$-Laplacian 
	with a Steklov boundary condition. In particular they proved that
	$$ 
		\lim_{p\to+ \infty}\left(\lambda_{1,p}^S\right)^{1/p} = 1 
		\quad \text
		{ and }		\quad
		\lambda_{2,\infty}^{S}\coloneqq\lim_{p\to +\infty}
		\left(\lambda_{2,p}^S\right)^{1/p} 
		= \frac{2}{\diam(U)}, 
	$$
	and also identify the limit variational problem defining 
	$\lambda_{2,\infty}^S$.
	
	\medskip

	The paper is organized as follows: in Section \ref{sect-MK} we collect some 
	preliminary results concerning optimal mass transport with cost 
	$d(x,y)^s$, in particular, we provide a statement of the 
	Kantorovich duality result that will be used in the proofs of our 
	results; in Section \ref{sect-Dir} we deal with the 
	Dirichlet problem and prove Theorem \ref{teo.Dir.Intro}; 
	in Section \ref{sect-New} we study the Neumann case 
	(Theorem~\ref{teo.Neu.Intro}). Finally, in Section \ref{elOtro} we deal
	with problem \eqref{eq:naut.int} 
	and we prove Theorem \ref{teo:oautovalointro}.

\section{Kantorovich duality for the cost $c(x,y)=d(x,y)^s$ }\label{sect-MK}

	In this section we follow \cite{Villani}. We first recall the definition of $c$-concavity and $c$-transform.

	\begin{defi}[{\cite[Definitions 5.2 and 5.7]{Villani}}]
		Let $X,Y$ be two sets and $c:X\times Y\to \R\cup\{+\infty\}.$
		A function $\psi:X\to \R\cup\{+\infty\}$ is said to be 
		$c$-convex if $\psi\not\equiv +\infty$ and there exists 
		$\phi:Y\to \R\cup\{\pm\infty\}$ such that
		\begin{equation}\label{Defcconvex} 
			\psi(x)=\sup_{y\in Y} \phi(y)-c(x,y) \qquad \text{ for all } 
			x\in X.
		\end{equation} 
		Its $c$-transform is the function $\psi^c$ defined 
		by
		$$ 
			\psi^c(y)=\inf_{x\in X} \psi(x)+c(x,y)\qquad 
			\text{ for all } y\in Y. 
		$$
		A function $\phi:Y\to \R\cup\{-\infty\}$ is $c$-concave 
		if $\phi\not\equiv -\infty$ and $\phi=\psi^c$ for some function 
		$\psi:X\to \R\cup\{\pm\infty\}$. Then its c-transform $\phi^c$ is
 		$$ 
 			\phi^c(x)=\sup_{y\in Y} \phi(y)-c(x,y)\qquad \text{ for all } 
 			x\in X. 
 		$$
	\end{defi}

	There holds: 

	\begin{prop}[{\cite[Proposition 5.8]{Villani}}] 
		For any $\psi:X\to \R\cup\{+\infty\}$, $\psi^c=\psi^{ccc}$ and $\psi$ is $c$-convex 
		iff $\psi=\psi^{cc}$.
	\end{prop}

	In the case where the cost function is $c(x,y)=d(x,y)^s$, we have the following characterization of $c$-convex function.

	\begin{lema}\label{concave} 
		Let $c(x,y)=d(x,y)^s$  and $X=Y=\overline{U}.$ 
		Any c-convex function $\psi$ satisfies $\psi^c=\psi$ and 
		\begin{equation}\label{sHolder}
			|\psi(x)-\psi(\tilde x)|\le d(x,\tilde x)^s \qquad 
			\text{ for all } x,\tilde x\in\overline{U}.
		\end{equation}
	\end{lema}

	\begin{proof}
	Notice that 
		\[
			\psi^c(y)=\inf_{x\in\overline{U}} \psi(x)+d(x,y)^s
			\le \psi(y)
		\] 
		and that the opposite inequality holds if \eqref{sHolder} holds.
    We now verify \eqref{sHolder}. Let $\phi:\overline{U}\to \R\cup\{\pm\infty\}$ such that $\psi=\phi^c$ as in \eqref{Defcconvex}. 
    Since $s\in (0,1)$, we have $d(x,y)^s\le d(x,\tilde x)^s+d(y,\tilde x)^s $ for any 
		$x,\tilde x,y\in\overline{U}$.
		It follows that
		\begin{eqnarray*}
			\psi(x) &=& \phi^c(x) = \sup_{y\in\overline{U}} \phi(y)-d(x,y)^s  \\
				&\ge& \sup_{y\in \overline{U}} 	\phi(y)-d(y,\tilde x)^s - d(x,\tilde x)^s
				= \phi(\tilde x)- d(x,\tilde x)^s,
		\end{eqnarray*}
		i.e. $\phi(\tilde x)-\phi(x)\le d(x,\tilde x)^s$. The opposite 
		inequality holds as well by switching $x$ and $\tilde x$.
		Thus \eqref{sHolder} holds.
	\end{proof}

	We recall the following result, see \cite[Theorem 5.9]{Villani}.

	\begin{teo}
		Let $(X,\mu)$ and $(Y,\nu)$ be two Polish probability spaces 
		(i.e. metric complete separable) and let 
		$c:X\times Y\to \R\cup\{+\infty\}$ be a lower 
		semicontinuous function such that
		$$ 
			c(x,y)\ge a(x)+b(y) \qquad \text{ for all } 
			(x,y)\in X\times Y 
		$$
		for some real-valued upper semicontinuous 
		functions $a\in L^1(\mu)$ and $b\in L^1(\nu)$. 
		Then, letting $J(\phi,\psi):=\int_Y\phi \,d\nu - \int_X \psi\,d\mu$, 
		\begin{eqnarray*}  
			W_c(\mu,\nu):=\min_{\pi \in\Pi(\mu,\nu)} \int_{X\times Y} c(x,y)\,d\pi(x,y)
			&=& \sup_{(\psi,\phi)\in L^1(\mu)\times L^1(\nu),\,\phi-\psi\le c} J(\phi,\psi) \\			  
			&=& \sup_{\psi\in L^1(\mu)} J(\psi^c,\psi)
		\end{eqnarray*} 
		and in the above $\sup$, one might as well impose $\psi$ to be $c$-convex.
		Moreover if $c$ is real-valued, $W_c(\mu,\nu)<\infty$ and
		$$ 
			c(x,y)\le c_X(x)+c_Y(y) \qquad \text{ for all } 
			(x,y)\in X\times Y  	
		$$
		for some $c_X\in L^1(\nu)$ and $c_Y\in L^1(\mu)$, then the above 
		$\sup$ is a $\max$ and one might as well impose $\psi$ to be $c$-convex.
	\end{teo}

	In the particular case $c(x,y)=d(x,y)^s$, $X=Y=\overline{U}$ 
	with $U$  bounded, we obtain in view of 
	Lemma \ref{concave} the following result.

	\begin{teo}\label{Duality} 
		For any $\mu,\nu\in P(\overline{U})$,
		\begin{eqnarray*}
 			\min_{\pi \in\Pi(\mu,\nu)} 
 			\int_{\overline{U}\times\overline{U}}  d(x,y)^s\,d\pi(x,y)
				&=& \max_{|\psi (x) - \psi (y)|\le d(x,y)^s } 
				\int_{\overline{U}}\psi\,d\nu - \int_ {\overline{U}} \psi\,d\mu .
		\end{eqnarray*}
	\end{teo}

	\begin{proof} 
	In view of Lemma \ref{concave}  and the previous theorem, we can write that 
	\begin{eqnarray*} 
	W_c(\mu,\nu) 
	&=& \max_{\psi\in L^1(\mu)\,\mbox{\tiny c-convex}} J(\psi^c,\psi) \\
	&\le& \max_{|\psi(x)-\psi(y)|\le d(x,y)^s} J(\psi,\psi) \\
	&\le& \sup_{(\psi,\phi)\in L^1(\mu)\times L^1(\nu),\,\phi-\psi\le d(x,y)^s} J(\phi,\psi) \\
	&	=& W_c(\mu,\nu) 	  
	\end{eqnarray*}
	form which we deduce the result.
	\end{proof}

\section{The Dirichlet case} \label{sect-Dir}
	In this section, we borrow ideas from \cite{CDJ}.
	Let us consider
	\[
		G_p, G_\infty\colon C(\overline{U})\times 
		M(\overline{U})\to \R\cup\{+\infty\} 
	\] 
	the functionals given by
	\begin{equation*}
 		G_p(v,\sigma) =
 				\begin{cases}
  					-\displaystyle\int_U v \sigma \, dx,
  					&\text{ if } \sigma\in L^{p'}(U),\, 
  					\|\sigma\|_{L^{p'}(U)}\le 1, \\
   					& \text{ and } v\in 
   				\widetilde{W}^{s,p}(U),\, [v]_{s,p}\le 
   				(\lambda_{s,p}^D)^{1/p}, \\
  				+ \infty & \text{ otherwise, }
				\end{cases}
	\end{equation*}
	and
	\begin{equation*}
 		G_\infty(v,\sigma) =
 			\begin{cases}
  				-\displaystyle\int_U v\,d\sigma, &\text{ if } 
  					\sigma\in M(\overline{U}),\, 
  					|\sigma|(U)\le 1, \\
  								&\text{ and } v\in 
  								\widetilde{W}^{s,\infty}(U),
  								\, | v(x) - v(y) 
  								|\le \Lambda^D_{s,\infty} |x-y|^s, \\
  					+ \infty & \text{ otherwise. }
			\end{cases}
	\end{equation*}
	
	In the space $M(\overline{U})$, we consider the weak convergence of 
	measures, and in the space $C(\overline{U})$ the uniform convergence. 

	First, we have that $G_\infty$ is the limit as $p\to+ \infty$ of 
	$G_p$ in the $\Gamma-$limit sense 
	(we refer to \cite{dalmaso} for the definition of 
	$\Gamma-$convergence).

	\begin{lema} \label{lema.gamma.conver} 
		The functionals $G_p$ $\Gamma-$converge as 
		$p\to +\infty$ to $G_\infty$.
	\end{lema}
	\begin{proof}
		It follows as in \cite{CDJ}.
	\end{proof}

	Now, we let $f_p\colon\mathbb{R}^n\to\mathbb{R}$ defined as
	$$
	    f_p(x) \coloneqq (u_p(x))^{p-1},
	$$
	where $u_p$ is a nonnegative eigenfunction associated to $\lambda_{s,p}^D(U)$
	such that $\|u_p\|_{L^p(U)}=1.$	
	When we consider $f_p$ 
	as an element of $M(\overline{U})$ together with $u_p$ 
	we obtain a minimizer for $G_p$. The proof of this fact is immediate.

	\begin{lema} \label{lema.fp.minimiza}
		The pair $(f_p,u_p)$ minimizes $G_p$ in 
		$C(\overline{U})\times M(\overline{U})$ with
		$$
			G_p (f_p, u_p) =-1.
		$$
	\end{lema}

	Now, let us show that we can extract a subsequence $p_n \to +\infty$ 
	such that $f_p$ and $u_p$ converge.

	\begin{lema} \label{lema.fp.converge} 
		There exists a sequence $p_n \to +\infty$ such that
		$$
			u_{p_n} \to u_\infty
		$$
		uniformly in $\R^n$. 
		This limit $u_\infty$ verifies
		$$
			| u_\infty (x) - u_\infty (y) |\le 
			\Lambda^D_{s,\infty} |x-y|^s, \quad x,y\in\R^n.
		$$
		Moreover, we have
		$$
			f_{p_n}  \stackrel{*}{\rightharpoonup}  f_\infty
		$$
		weakly-* in $ M(\overline{U})$ and $f_\infty$ is a  nonnegative 
		measure that verifies $f_\infty(\overline{U}) \le 1$.
	\end{lema}

	\begin{proof}
		The convergence of $u_{p},$ via a subsequence, 
		is contained in \cite{LL}. 
		Concerning $f_{p_n}$ the conclusion follows 
		from the inequality
		\begin{equation}\label{pp}
			\int_U f_{p}\, dx \leq 
			\left( \int_U (u_{p})^{p}\, dx \right)^{\frac{p-1}{p}} 
			|U|^{\nicefrac1{p}} = |U|^{\nicefrac1{p}},
		\end{equation}
		that implies that $f_{p}$ is bounded in 
		$M(\overline{U})$ and hence we can extract a sequence 
		$p_n \to+ \infty$
		such that $f_{p_n} \stackrel{*}{\rightharpoonup} f_\infty$
		weakly-* in $ M(\overline{U})$. 
		That the limit $f_\infty$ is a nonnegative measure that 
		verifies $f_{\infty}(\overline{U}) \leq 1$ also follows from \eqref{pp}.
	\end{proof}

	From the main property of $\Gamma-$convergence we obtain the 
	following corollary.

	\begin{cor} \label{finfty.minimiza}
		The pair $(f_\infty, u_\infty)$ minimizes $G_\infty$ with
		$$	
			G_\infty (f_\infty, u_\infty)=-1.
		$$
	\end{cor}

	Now we are ready to prove Theorem \ref{teo.Dir.Intro}.

	\begin{proof}[Proof of Theorem \ref{teo.Dir.Intro}]
		As $(f_\infty, u_\infty)$ minimizes $G_\infty$ we obtain that
		$$
			\Big(f_\infty, \frac{u_\infty}{\Lambda^D_{\infty,s}} \Big)
		$$ 
		minimizes
		$$
		-\int_U v\,d\sigma,
		$$
		with $(v,\sigma)$ belonging to
		$$
			A\coloneqq
			\left\{ (v,\sigma)\in 
			\widetilde{W}^{s,\infty}(U)\times
			M(\overline{U})
			\colon|\sigma|(U)\le 1,
			 | v(x) - v(y) |\le d(x,y)^s 
			\right\}.
		$$
		Then
		$$
			\begin{array}{rl}
				\displaystyle
				\frac{1}{\Lambda^D_{s,\infty}} &
				\displaystyle = \frac{1}{\Lambda^D_{s,\infty}} 
				\int_U u_\infty\, df_\infty \\[10pt]
				& \displaystyle = \max_{(v,\sigma)\in A}
				 \int_U v\,d\sigma  \\[10pt]
				& \displaystyle = \max_{\mu\in P(\overline{U})} 
				\max_{| w(x) - w(y) |\le d(x,y)^s} \int_U w \, d\mu \\[10pt]
				& \displaystyle = \max_{\mu\in P(\overline{U})} 
				W_s (\mu, P(\partial U)),
			\end{array}
		$$
		as we wanted to show.
\end{proof}

\section{The Neumann case} \label{sect-New}

	Again, we follow ideas from \cite{CDJ}, see also \cite{RoSaint}.
	Let $u_p$ be an extremal for $\lambda^N_{p,s}$ 
	(that is, a minimizer for \eqref{DefLambdapsNeu})
	normalized by $\|u_p\|_{L^p(U)}=1$.
	Then $f_p\coloneqq |u_p|^{p-2}u_p \in L^{p'}(U)$ (where $p'=\frac{p}{p-1}$)
	satisfies
	\begin{equation}\label{Propfp}
 		\|f_p\|_{L^{p'} (U)}=1 \quad \text{and} \quad \int_U f_p \, dx =0.
	\end{equation}
	The first step consists in extracting from $\{f_p\}_{p>1}$ 
	a subsequence converging weakly to some measure 
	$f_\infty \in M(\overline{U})$, 
	the weak convergence meaning here that
	$$
		\lim_{p\to +\infty}\int_{\overline{U}}\phi f_p\,dx = 
		\int_{\overline{U}}\phi\,df_\infty$$
	for any $\phi\in C(\overline{U})$.

	\begin{lema}  
		Up to a subsequence, the measures $f_p$ converge weakly in 
		measure in $\overline{U}$ to some measure $f_\infty$ supported in 
		$\overline{U}$ satisfying
		\begin{equation}\label{Propfinfinity}
 			f_\infty(\overline U) = 0\quad \text{ and } \quad 
 			|f_\infty|(\overline U) = 1.
		\end{equation}
	\end{lema}
	\begin{proof}
		We claim that
		\begin{equation}\label{Propfp2}
 			\lim_{p\to +\infty} \int_U |f_p|\,dx = 1.
		\end{equation}
		First, in view of \eqref{Propfp}, we have that
		$$ 
			\int_U |f_p|\,dx \le \|f_p\|_{L^{p'}(U)}|U|^{1-1/p'}=|U|^{1-1/p'}
			\to 1 \qquad \text{as } p\to +\infty
		 $$
		and then, recalling that $u_p\to u$ in $C(\overline{U})$
		with $\|u\|_{L^\infty(U)}  =1$,
		$$ 
			1 = \int_U u_p f_p\,dx \le 
			\|u_p\|_{L^\infty(U)} \|f_p\|_{L^1(U)} 
			= (1+o(1))\|f_p\|_{L^1(U)}. 
		$$

		It follows in particular that the measures $|f_p|$ are
		bounded in $M(\overline{U})$ independently of $p$. 
		Since $\overline{U}$ is compact, 
		we can then extract from this sequence a subsequence converging 
		weakly to some measure $f_\infty\in M(\overline{U})$. Passing to 
		the limit in \eqref{Propfp} and \eqref{Propfp2} gives 
		\eqref{Propfinfinity}.
	\end{proof}

	Consider the functionals 
	$G_p, G_\infty\colon C(\overline{U})\times M(\overline{U})\to 
	\R\cup\{+\infty\} $ defined by
	\begin{equation*}
 		G_p(v,\sigma) =
 		\begin{cases}
  			\displaystyle -\int_U v \sigma\, dx &\text{ if } 
  			\sigma\in L^{p'}(U),\, 
  			\|\sigma\|_{L^{p'}(U)}\le 1,\, \int_U \sigma\,  dx = 0, \\
   			&\text{ and }  
   				v\in W^{s,p}(U),\, \llbracket v\rrbracket_{s,p}
   				\le (\lambda_{p,s}^N)^{1/p}, \\
  					+\infty & \text{ otherwise, }
		\end{cases}
	\end{equation*}
	and
	\begin{equation*}
 		G_\infty(v,\sigma) =
 		\begin{cases}
  			\displaystyle  -\int_U v\,d\sigma,& \text{ if } 
  			\sigma\in M(\overline{U}),\, 
  			|\sigma|(\overline{U})\le 1,\, \sigma(\overline{U}) = 0, \\
   			& \text{ and } 
   			v\in W^{s,\infty}(U),\, 
   			| v(x) - v(y) |\le \Lambda^N_{\infty,s} d(x,y)^s, \\
  				+\infty & \text{ otherwise. }
		\end{cases}
	\end{equation*}
	
	Remark that these functionals are similar to the ones considered for the Dirichlet case but the spaces involved change. In fact, here we consider $W^{s,p}(U)$ instead of $\widetilde{W}^{s,p}(U)$ (that encodes the fact that we are considering functions that vanish outside $U$ when dealing with the Dirichlet problem).

As for the Dirichlet case,
	we can prove as in \cite{CDJ,RoSaint} that
	$G_\infty$ is the limit of the $G_p$ in the sense of $\Gamma$-convergence:

	\begin{lema}
		The functionals $G_p$ converge in the sense of $\Gamma$-convergence to 
		$G_\infty$.
	\end{lema}

	The proof is similar as that of Proposition 3.7 in \cite{CDJ} 
	and hence we omit it.	As a corollary we obtain that

	\begin{lema} 
		Let $u_p$ be an extremal for $\lambda^N_{p,s}$, then 
		$(u_p,f_p)$ is a minimizer for $G_p$, and any limit 
		$(u_\infty,f_\infty)$ along a subsequence 
		$p_j\to+ \infty$  is a minimizer for 
		$G_\infty$, with
		$$	G_\infty(u_\infty,f_\infty) = \lim_{p\to+ \infty} G_p(u_p,f_p) 
 			= -1.
		$$
	\end{lema}
	\begin{proof}
		Notice that the pair $(u_p,f_p)$ is a minimizer of $G_p$. Indeed, given a 
		pair $(v,\sigma)$ admissible for $G_p$ take $\bar v\in\R$ such that $$
		\int_U |v-\bar v|^{p-2}(v-\bar v)\,dx=0.$$
		Then, recalling that $\int_U\sigma\, dx=0$ and the definition of 
		$\lambda^N_{p,s}$, we have
		\begin{equation*}
		\begin{split}
	 		G_p(v,\sigma) & = -\int_U (v-\bar v)\sigma\, dx  \\
			& \ge - 
	 		\|v-\bar v\|_{L^p(U)}\|
	 		\sigma\|_{L^{p'}(U)}    \\
	        & \ge -(\lambda_{p,s}^N)^{-1/p} \llbracket v-\bar v\rrbracket_{s,p}   \\
	        & \ge -1 = G_p(u_p,f_p).
		\end{split}
		\end{equation*}
		Moreover $(u_p,f_p)\to (u_\infty,f_\infty)$ along a sequence 
		$p_j \to+ \infty$. Then, it follows that
		$$ 
			\liminf_{p\to +\infty} (\inf\,G_p) = \liminf_{p\to +\infty}
				G_p(u_p,f_p)\ge G_\infty(u_\infty,f_\infty)\ge \inf_B
				G_\infty 
		$$
		where $B$ is the set of all pairs 
		$(v,\sigma)\in W^{s,\infty}(U)\times M(\overline{U})$
		such that
		\[
			|\sigma|(\overline{U})\le1,\, \sigma(\overline{U})=0, 
			\text{ and }
			|v(x)-v(y) |\le
			\Lambda^N_{\infty,s} d(x,y)^s.
		\]
		Moreover, the $\limsup$ property implies that
		$$
			\limsup_{p \to+ \infty} \left(\inf_B\,G_p\right) 
			\le \inf_B\,G_\infty.
		$$
		Hence
		$$ 
			\lim_{p\to+\infty}\inf_B G_p=\lim_{p\to\infty} G_p(u_p,f_p)=
			G_\infty(u_\infty,f_\infty)
   			= \inf_B G_\infty.
   		$$
	\end{proof}

	We can now relate $\Lambda_{s,\infty}^N$ to $W_s$. 
	Recall that if $\sigma\in M(\overline{U})$, 
	then $\sigma^\pm\in M(\overline{U})$ denote the positive and 
	negative part of $\sigma$. In particular, $\sigma=\sigma^+ - 
	\sigma^-$, and $|\sigma|=\sigma^++\sigma^-$.

	\begin{proof}[Proof of Theorem \ref{teo.Neu.Intro}]
		The conditions $\sigma(\overline{U})=0$ and $|\sigma|(\overline{U})=1$ are 
		equivalent to 
		$$
			\sigma^+(\overline{U})=\sigma^-(\overline{U})=1/2.
		$$
		We can therefore rewrite the fact that the pair $(u_\infty,f_\infty)$ 
		is a minimizer of $G_\infty$ as
		\begin{equation*}
 			1  = \max_{\sigma\in M_{\nicefrac12}}\,
   			\max_{v\in 
   			F_{\Lambda^N_{s,\infty}}} \int_U v\, d(\sigma^+-\sigma^-),
		\end{equation*}
		where
		\begin{align*}
			 M_t&=\left\{\sigma\in M(\overline{U})\colon 
		 		\sigma^+(U)=\sigma^+(U)=t\right\},\\
		\end{align*}
		and
		\begin{align*}
		 	F_R&=\left\{v\in W^{s,\infty}(U)\colon |v(x)-v(y)| \le 
   			R\, d(x,y)\right\},
		\end{align*}
		that is,
		\begin{equation*}
 			\dfrac{2}{\Lambda^N_{\infty,s}}
 			= \max_{\sigma\in M_1} 
   			\max_{v\in F_1} 
   			\int_U v\, d(\sigma^+-\sigma^-).
		\end{equation*}
		Then, we obtain the conclusion \eqref{MK.Neu},
 recalling the definition of 
		$W_s$ given by \eqref{DefW1}.
	\end{proof}

\section{Eigenvalue problems with a different Neumann boundary condition}\label{elOtro}
	In this section we prove Theorem \ref{teo:oautovalointro}. 
	For this purpose, first we present some previous results.

	\begin{teo}\label{teo:space}
	The spaces 
		\[
			\mathcal{W}^{s,p}(U)\coloneqq
			\left\{u\colon\mathbb{R}^n\to\mathbb{R} \text{ measurable }	
			\colon
			\|u\|_{L^p(U)}^p+\mathcal{H}_{s,p}(u)<+\infty \right\}
		\]
		and
		\[
			\mathcal{W}^{s,\infty}(U)\coloneqq
			\left\{u\colon\mathbb{R}^n\to\mathbb{R} \text{ measurable }	
			\colon
			\|u\|_{L^\infty(U)}+\mathcal{H}_{s,\infty}(u)<+\infty \right\}
		\]
 are Banach spaces with the norms 
 $$			\|u\|_{\mathcal{W}^{s,p}(U)}^p\coloneqq
			\|u\|_{L^p(U)}^p+ \mathcal{H}_{s,p}(u)
$$	
and
$$
			\|u\|_{\mathcal{W}^{s,\infty}(U)}\coloneqq
			\|u\|_{L^\infty(U)}+ \mathcal{H}_{s,\infty}(u),
$$ 
respectively.
	\end{teo}
	
	The proof follows exactly as in the proof of 
		\cite[Proposition 3.1]{DRV}. 
	
	\begin{remark}
		It holds that
		$\mathcal{W}^{s,p}(U)\subset W^{s,p}(U).$
	\end{remark}
	\begin{remark} 
		The operator $I\colon\mathcal{W}^{s,p}(U)\to E=
			L^p(U)\times L^p(\mathbb{R}^{2n}
			\setminus (U^c)^2)$ given by
		\begin{align*}
			I(u)&\coloneqq\left(u, 
			\dfrac{u(x)-u(y)}{d(x,y)^{\frac{n}p+s}}\right)
		\end{align*}
		is an isometry. Then $I(\mathcal{W}^{s,p}(U))$ is a closed 
		subspace of 
		$E$ due to the fact that $\mathcal{W}^{s,p}(U)$
		is a Banach space. Hence $I(\mathcal{W}^{s,p}(U))$ is reflexive 
		since $E$ is reflexive. Then, 
		$\mathcal{W}^{s,p}(U)$ is reflexive.
	\end{remark}

	Following the proofs of Lemmas 3.2 and 3.7 in \cite{DRV}, 
	we have the following result.

	\begin{lema}\label{lema:partes} 
	Let $u$ and $v$ be bounded $C^2$ functions in $\mathbb{R}^n.$ 
	Then	 the following formulae hold:
	\begin{myitemize}
		\item Divergence theorem
		\[
			\int_{U}(-\Delta)_p^s u(x)\, dx=
			-\int_{\mathbb{R}\setminus U}\mathcal{N}_{s,p}u(x)\,
			dx.
		\]
		\item  Integration by parts formula
		\[			
			\frac1{2}\mathcal{H}_{s,p}(u,v)=\int_{U}v(x) (-\Delta)_p^s 
			u(x)\, dx
			+\int_{\mathbb{R}\setminus U} v(x)\mathcal{N}_{s,p}u(x)\, dx,
		\]
	\end{myitemize}
	where
		\[
				\mathcal{H}_{s,p}(u,v)\coloneqq \int
				\int_{\mathbb{R}^{2n}\setminus{(U^c)^2}}
				\dfrac{|u(x)-u(y)|^{p-2}(u(x)-u(y))(v(x)-v(y))}{d(x,y)^{n+ps}}
				\, dxdy.
		\]
	\end{lema}

	This result leads us to the following definition. 
	
	\begin{defi}
		A function $u\in \mathcal{W}^{s,p}(U)$ 
		is a weak solution of \eqref{eq:naut.int} if 
		\begin{equation}\label{eq:waut}
			\frac1{2}\mathcal{H}_{s,p}(u,v)=\lambda
			\int_{U}|u|^{p-2}u v\, dx
		\end{equation}
		for all $v\in \mathcal{W}^{s,p}(U).$
	\end{defi}
	
	In this context we have the following definition.
	
	\begin{defi}
	We say that $\lambda$ is a fractional Neumann $p-$eigenvalue 
	provided there exists a nontrivial weak solution 
	$u\in\mathcal{W}^{s,p}(U)$ of \eqref{eq:naut.int}. The function $u$ 
	is a corresponding eigenfunction.
	\end{defi}
	
	Let us observe the following: if $\lambda>0$ is an eigenvalue
	and $u$ is an eigenfunction associated to $\lambda,$
	then, taking $v\equiv1$ as a test function in 
	\eqref{eq:waut}, we have
	\[
		\int_{U}|u|^{p-2}u\, dx = 0.
	\]

	In fact, we have that $\lambda=0$ is the first eigenvalue of our problem.
	
	\begin{lema}
		It holds that $\lambda=0$ is an eigenvalue of \eqref{eq:naut.int} 
		(with $u=1$ as eigenfunction), and it is isolated and simple. 
	\end{lema}
	\begin{proof}
		Let $u$ be an eigenfunction corresponding to $\lambda=0$ in problem 
		\eqref{eq:naut.int}. From \eqref{eq:waut} taking $v=u$ 
		as a test function we obtain that $u$ is constant in $U$. 
	
		Now, if we have a sequence of eigenvalues $\lambda_k \to 0$ then the 
		corresponding eigenfunctions,  $u_k$, normalized by 
		$\|u_k\|_{L^p (U)}=1$, 
		converge to some $u$. It is not difficult to show that $u$ is an 
		eigenfunction corresponding to $\lambda=0$ (consequently, 
		$u\equiv const$) 
		with $\|u\|_{L^p (U)}=1$ and $\int_U |u|^{p-2}u \, dx=0$, a 
		contradiction that shows that $\lambda=0$ is an isolated eigenvalue. 
	\end{proof}

		Thus, the existence of the first non-zero eigenvalue 
	of \eqref{eq:naut.int} is related to the 
	problem of minimizing the following non-local quotient
	\[
		\frac{\mathcal{H}_{s,p}(v)}{2\|v\|^p_{L^p(U)}}
	\]
	among all functions $v\in \mathcal{W}^{s,p}(U)\setminus\{0\}$ such 
	that
	$\int_{U}|v|^{p-2}v\,dx=0.$

	We are now ready to prove Theorem \ref{teo:oautovalointro}.
	For simplicity, we divide the proof of this theorem into three parts 
	contained in the following lemmas.

	First, by a standard compactness argument and using that 
	$\mathcal{W}^{s,p}(U)\subset W^{s,p}(U),$ we have that
	$\lambda_{s,p}$ is the first non-zero eigenvalue of \eqref{eq:naut.int}.

	\begin{lema}\label{lema:propaut}
     		 It holds that $\lambda_{s,p}$
      		is the first non-zero eigenvalue of \eqref{eq:naut.int}. 
	\end{lema}
	
	\begin{remark}\label{remark:estima} Since 
		$\mathcal{W}^{s,p}(U)\subset W^{s,p}(U)$ 
		and
		\[
			\llbracket u \rrbracket_{s,p}^p\le \mathcal{H}_{s,p}(u),
			\qquad 	
			\forall u\in \mathcal{W}^{s,p}(U),
		\]
		we have that
		\[
			\lambda_{s,p}^N\le 2\lambda_{s,p}.
		\]
	\end{remark}

	Our next result shows the asymptotic behavior of 
	$(\lambda_{s,p})^{\nicefrac{1}{p}}.$

	\begin{lema}\label{lema:anv}
		We have
		\[
			\lim_{p\to+ \infty}(\lambda_{s,p})^{\nicefrac1p}=
			\dfrac{2}{(\diam_d(U))^s}=\Lambda_{s,\infty}\coloneqq
			\inf\left\{\dfrac{\mathcal{H}_{s,\infty}(u)}
			{\|u\|_{L^{\infty}(U)}}\colon u\in\mathcal{A}\right\},
		\]
		where 
		\[
			\mathcal{A}\coloneqq\left\{v\in\mathcal{W}^{s,\infty}(U)
			\setminus\{0\}
			\colon \displaystyle\sup_{x\in U} u(x) +\inf_{x\in U} u(x)=0
			\right\}.
		\]
	\end{lema}
	
	\begin{proof}
		For the reader's convenience, we split the proof in four steps.
		
		\noindent {\it Step 1.} We start showing that 
		\[
			\Lambda_{s,\infty}\le\dfrac{2}{(\diam_d(U))^s}.
		\]
		Let $x_0,y_0\in\overline{U}$ such that $d(x_0,y_0)=\diam_d(U).$ 
		Let $u\colon\mathbb{R}^n\to\mathbb{R}$ be given by
		\[
			u(x)\coloneqq -1 + \frac{2}{\diam_d(U)} d(x,y_0)^s.
		\]
		Observe that, $$\displaystyle\sup_{x\in U} u(x) 
		=-\inf_{x\in U} u(x)=1$$ and
		\[
			\dfrac{|u(x)-u(y)|}{d(x,y)^s}=
			\dfrac{2}{(\diam_d(U))^s}
			\dfrac{|d(x,y_0)^s-d(y,y_0)^s|}{d(x,y)^s}
			\le\dfrac{2}{(\diam_d(U))^s}
		\]
		for all $x,y\in\mathbb{R}^n.$
		Then $u\in\mathcal{A},$ $\|u\|_{L^\infty(U)}=1$ and
		\[
			\mathcal{H}_{s,\infty}(u)\le \dfrac{2}{(\diam_d(U))^s}.
		\]
		Therefore
		\[
			\Lambda_{s,\infty}\le\mathcal{H}_{s,\infty}(u)
			\le \dfrac{2}{(\diam_d(U))^s}.
		\]
		\noindent {\it Step 2.} We now prove that 
		\[
			\Lambda_{s,\infty}\ge\dfrac{2}{(\diam_d(U))^s}.
		\]
		If $u\in\mathcal{A}$ then
		\begin{align*}
			2\|u\|_{L^{\infty}(U)}&=\sup_{x\in U} u(x)
										-\inf_{x\in U} u(x)\\
									&=\sup\left\{|u(x)-u(y)|\colon
									x,y\in U \right\}\\
									&\le (\diam_d(U))^s
									\sup\left\{
									\dfrac{|u(x)-u(y)|}{d(x,y)^s}\colon
									x,y\in U \right\}\\
									&\leq (\diam_d(U))^s
									\mathcal{H}_{s,\infty}(u).
		\end{align*}
		Thus
		\[
			\dfrac{2}{(\diam_d(U))^s}\le
			\dfrac{\mathcal{H}_{s,\infty}(u)}{\|u\|_{L^{\infty}(U)}}
		\]
		for any $u\in\mathcal{A},$ that is,
		\[
			\Lambda_{s,\infty}\ge\dfrac{2}{(\diam_d(U))^s}.
		\]
		
		\noindent {\it Step 3.} We show that
		\[
			\dfrac{2}{(\diam_d(U))^{s}}\le
			\liminf_{p\to+\infty}( \lambda_{s,p})^{\nicefrac{1}{p}}.
		\]
		By \eqref{eq:1ercaso} and Remark \ref{remark:estima}, we have 
		that
		\[ 
		\begin{array}{l}
		\displaystyle 
			\dfrac{2}{(\diam_d(U)^{s})}
			\le\lim_{p\to+\infty} \left(\lambda_{s,p}^N\right)
			^{\nicefrac{1}{p}}
			 \\[10pt]
			 \qquad \qquad \qquad \displaystyle \le
			\liminf_{p\to+\infty} 2^{\nicefrac1p}
			(\lambda_{s,p}(U))^{\nicefrac{1}{p}}
			=\liminf_{p\to+\infty} 
			(\lambda_{s,p}(U))^{\nicefrac{1}{p}}.
			\end{array}
		\]     	
		
		\noindent {\it Step 4.} Finally, we prove that
		\[
			\limsup_{p\to+\infty} (\lambda_{s,p})^{\nicefrac{1}{p}}
			\le\dfrac{2}{(\diam_d(U))^{s}}.
		\]
		As in Step 1, let $x_0,y_0\in \overline{U}$ be such that
		$d(x_0,y_0)=\diam_d(U).$ Set $\delta=\diam_d(U),$
		\[
			U_{\delta}\coloneqq\left\{x\in\mathbb{R}^n
			\colon \inf_{y\in U} d(x,y)\le \delta\right\}
		\]
		and
		\[
			u(x)\coloneqq
			\begin{cases}
					d(x,y_0) &\text{ if } x\in U_{\delta},\\
					0 &\text{ if } x\in \mathbb{R}^n\setminus
					U_{\delta}.
			\end{cases}
		\]
		
		Let $\varepsilon>0.$ Then
		\begin{align*}
			\mathcal{H}_{s,p}(u)\!\le&2
			\int_{U\times U_{\delta}}\!\!\!\!\!\!
			\dfrac{|d(x,y_0)-d(y,y_0)|^p}{d(x,y)^{n+sp}}
			\, dxdy \\
			& \qquad + 
			2\int_{U\times (\mathbb{R}^n\setminus U_{\delta})}
			\dfrac{d(x,y_0)^p}{d(x,y)^{n+sp}}
			\, dxdy\\
			\le& 2
			\int_{U\times U_{\delta}}\!\!\!\!\!\!
			\dfrac{d(x,y)^{p(1-s)-\varepsilon}}{d(x,y)^{n-\varepsilon}}
			\, dxdy 
			\\
			& \qquad + 
			2\int_{U\times (\mathbb{R}^n\setminus U_{\delta})}
			\dfrac{d(x,y_0)^p}{d(x,y)^{n+\varepsilon+sp-\varepsilon}}
			\, dxdy.
		\end{align*}
		Thus, since $d$ is a distance equivalent to the Euclidean one,
		if $$p>\max\left\{
		\frac{\varepsilon}{(1-s)},
		\frac{\varepsilon}{s}\right\}$$ we get that
		$u\in\mathcal{W}^{s,p}(U)$ and
		\begin{equation}\label{eq:aux1}
			\mathcal{H}_{s,p}(u)
			\le C(\diam_d(U))^{p(1-s)}
			\left\{(\diam_d(U))^{-\varepsilon}+ 
			(\diam_d(U))^{\varepsilon}\right\},
		\end{equation}
		where $C$ is a constant independent of $p.$ 
				
		We now choose $c_p\in\mathbb{R}$ such that
		\[
			w_p(x)=u(x)-c_p
		\]
		satisfies 
		\[
			\int_U|w_p|^{p-2}w_p\, dx=0.
		\]
		Hence, if $p>\max\left\{
		\nicefrac{\varepsilon}{(1-s)},
		\nicefrac{\varepsilon}{s}\right\},$ by \eqref{eq:aux1},
		we have that
		\begin{align*}
			\lambda_{s,p}&
			\le\dfrac{\mathcal{H}(w_p)}{2\|w_p\|_{L^p(U)}^p}\\
			&= \dfrac{\mathcal{H}(u)}{2\|w_p\|_{L^p(U)}^p}\\
			&\le \dfrac{C}{2\|w_p\|_{L^p(U)}^p}
			(\diam_d(U))^{p(1-s)}
			\left\{(\diam_d(U))^{-\varepsilon}+ 
			(\diam_d(U))^{\varepsilon}\right\},
		\end{align*}
		therefore
		\begin{equation}\label{eq:aux2}
			\limsup_{p\to+\infty}(\lambda_{s,p})^{\nicefrac1p}\le
			\dfrac{(\diam_d(U))^{1-s}}
			{\displaystyle\liminf_{p\to+\infty}\|w_p\|_{L^p(U)}}.
		\end{equation}
		
		On the other hand, in \cite{EKNT} it is proved that
		\begin{equation}\label{eq:aux3}
			\liminf_{p\to+\infty}\|w_p\|_{L^p(U)}\ge \dfrac{2}{\diam_d(U)}.
		\end{equation}
		
		Thus, by \eqref{eq:aux2} and \eqref{eq:aux3}, we get
		\[
			\limsup_{p\to+\infty} (\lambda_{s,p})^{\nicefrac{1}{p}}
			\le\dfrac{2}{(\diam_d(U))^{s}}.
		\]
		This concludes the proof.
	\end{proof}
	
	\begin{remark}\label{remark:compara}
	By  \eqref{eq:1ercaso} and Lemma \ref{lema:anv}, we have that
	\[
		\Lambda_{s,\infty}^N
		=\lim_{p\to+\infty}\left(\lambda_{s,p}^N\right)^{\nicefrac1p}
		=\frac{2}{(\diam_d(U))^s}
		=\lim_{p\to+\infty}(\lambda_{s,p})^{\nicefrac1p}
		=\Lambda_{s,\infty}.
	\]
	\end{remark}
	
	Concerning  the convergence as $p\to+\infty$ of the eigenfunctions we have the following result.
	
	\begin{lema}\label{lema:conver}
		If $u_p$ is a minimizer of $\lambda_{s,p},$ normalized with $\|u_p \|_{L^p (U)}=1$,
		then, up to a subsequence, $u_p$ converges in $C(\overline{U})$
		to some minimizer $u_{\infty}\in W^{s,\infty}(U)$ of 
		$\Lambda_{s,\infty}^N.$
	\end{lema}
	\begin{proof}
		For any $p\in(1,\infty),$ we consider
		$u_p\in\mathcal{W}^{s,p}(U)$ such that
		$$
			\|u_p\|_{L^p(U)}=1, \qquad \int_U |u_p|^{p-2}u_p\, dx=0 
			\qquad \mbox{and} \qquad
			\frac12\mathcal{H}_{s,p}(u_p)=\lambda_{s,p}.	
		$$
		Then, by Lemma \ref{lema:anv}, 
		there exists a constant $C$ independent of $p$ such that
		\begin{equation}
			\label{eq:caux3}
				\left(\frac{\mathcal{H}_{s,p}(u_p)}2
				\right)^{\frac1p}\le C
		\end{equation}
		for all $p\in(1,\infty).$
		
		Let us fix $q\in(1,\infty)$ such that $sq>2n.$ 
		If $p>q$ then, by H\"older's inequality, we have that
		\begin{equation}
			\label{eq:caux4}
			\|u_p\|_{L^q(\Omega)}\le |U|^{\frac1{q}-\frac1{p}}
			\|u_p\|_{L^{p}(\Omega)}\le |U|^{\frac1{q}-\frac1{p}}
			\quad\forall p\ge q, 
		\end{equation}
		and taking $r=s-\nicefrac{n}{q}\in(0,1),$ 
		again by H\"older's inequality, we get
		\begin{equation}
			\label{eq:caux5}
			\begin{aligned}
				\llbracket u_p\rrbracket_{r,q}^q
				&= \int_U\int_U\dfrac{|u_p(x)-u_p(y)|^q}{d(x,y)^{sq}}\, 
					dxdy\\
				&\le |U|^{2(1-\frac{q}{p})}\left(
				\int_U\int_U\dfrac{|u_p(x)-u_p(y)|^{p}}
				{d(x,y)^{sp}}\,
			 		dxdy\right)^{\frac{q}{p}}\\
				&\le 2^{\frac{q}p}(\diam_d(U))^{\frac{nq}{p}} 
				|U|^{2(1-\frac{q}{p})}
				\left(\frac{\mathcal{H}_{s,p}(u_p)}2
				\right)^{\frac qp}.
			\end{aligned}
		\end{equation}
		Then, by \eqref{eq:caux3}, we get
		\[
		\llbracket u_p\rrbracket_{r,q}
		\le 2^{\frac1p}(\diam_d(U))^{\frac{n}{p}} 
		|U|^{2\left(\frac1q-\frac1{p}\right)}C^q \qquad\forall p\ge q, 
	\]
	where $C$ is a constant independent of $p.$
	Hence $\{u_p\}_{p\ge q}$ is a bounded sequence in $W^{r,q}(U).$
	Then, since $rq=sq-n>n,$ by fractional compact embedding theorems 
	(see \cite[Theorem 4.54]{FD}), 
	there exist
	a function $u_{\infty}\in C(\overline{U})$ and
	a subsequence $\{u_{p_j}\}_{j\in\mathbb{N}}$ of $\{u_{p}\}_{p\ge q},$ 	
	such that
	\begin{align*}
		u_{p_j}\to u_{\infty} &\quad\mbox{uniformly in } 
		\overline{U},\\
		u_{p_j}\rightharpoonup u_{\infty} &\quad\mbox{weakly in }
		 W^{r,q}(U).
	\end{align*}
	Hence, by \eqref{eq:caux4}, 
	$\|u_\infty\|_{L^q(\Omega)}\le |U|^{\frac1q},$ and
	by \eqref{eq:caux5} and Remark \ref{remark:compara}, we get
	\begin{align*}
		\llbracket u_\infty\rrbracket_{r,q}&\le
		\liminf_{j\to \infty}\llbracket u_{p_j}\rrbracket_{r,q}\\
		&\le\liminf_{j\to \infty} 2^{\nicefrac{1}{p_j}}
			(\diam_d(U))^{\nicefrac{n}{p_j}} 
				|U|^{2\left(1-\nicefrac{1}{p_j}\right)}
				\left(\frac{\mathcal{H}_{s,p_j}(u_{p_j})}2
				\right)^{\nicefrac1{p_j}}\\
		&=|U|^{\frac2{q}}\Lambda_{s,\infty}^N.
	\end{align*}    
	    
	Letting $q\to+\infty,$ we obtain 
	$$\|u_\infty\|_{L^{\infty}(\Omega)}\le 1$$
	and
	\begin{equation}
	\label{eq:caux6}
		\llbracket u_\infty\rrbracket_{s,\infty}\le
		\Lambda_{s,\infty}^N.
	\end{equation}
	
	On the other hand,
	\[
		1=\|u_{p_j}\|_{L^{p_ j}(U)}\le|U|^{\frac1{p_j}}
		\|u_{p_j}\|_{L^{\infty}(U)}\qquad\forall j\in\mathbb{N},
	\]
	then $$1\le\|u_\infty\|_{L^{\infty}(U)}.$$ Hence 
	$$\|u_\infty\|_{L^{\infty}(U)}=1$$ and by \eqref{eq:caux6} 
	we get
	\begin{equation}
	\label{eq:caux7}
		\dfrac{\llbracket u_\infty\rrbracket_{s,\infty}}
		{\|u_\infty\|_{L^{\infty}(U)}}\le \Lambda_{s,\infty}^N.
	\end{equation}
	
	Finally, in \cite{EKNT,RoSaint} it was proved that the condition 
	$$\int_U |u_{p_j}|^{p_j-2} u_{p_j}\, dx =0$$ leads to
	$$\sup_{x\in U} u_{\infty}(x) +\inf_{x\in U} u_{\infty}(x)=0$$ 
	in the limit as $p\to +\infty$. Then, 
	using \eqref{eq:caux7}, we have that $u_\infty$ is a minimizer of 
	$\Lambda_{s,\infty}^N.$
	\end{proof}

\end{document}